\documentclass[reqno]{article}
\usepackage{amsmath,amsfonts,amssymb,amsthm,graphics,wrapfig}

\newtheorem{theorem}{Theorem}
\newtheorem{lemma}[theorem]{Lemma}
\newtheorem{corollary}[theorem]{Corollary}

\newcommand\eps{\varepsilon}
\renewcommand\le{\leqslant}
\renewcommand\ge{\geqslant}
\newcommand\floor[1]{\lfloor #1 \rfloor}
\newcommand\ceil[1]{\lceil #1 \rceil}

\newcommand\cS{{\mathcal S}}

\begin{document}
\title{Long cycles in random subgraphs of graphs with large minimum degree}
\author{Oliver Riordan%
\thanks{Mathematical Institute, University of Oxford, Radcliffe Observatory Quarter, Woodstock Road, Oxford OX2 6GG, UK.
E-mail: {\tt riordan@maths.ox.ac.uk}.}}
\date{August 16, 2013; revised May 24, 2014}
\maketitle

\begin{abstract}
Let $G$ be any graph of minimum degree at least $k$, and let $G_p$ be the random subgraph of $G$
obtained by keeping each edge independently with probability $p$.
Recently, Krivelevich, Lee and Sudakov showed that if $pk\to\infty$ then with probability
tending to 1 $G_p$ contains a cycle of length at least $(1-o(1))k$. We give a much shorter
proof of this result, also based on depth-first search.
\end{abstract}

Let $(G_k)_{k=1}^\infty$ be a sequence of graphs
where $G_k$ has minimum degree at least $k$ (we make no assumption
on the number of vertices of $G_k$), and let $0\le p=p(k)\le 1$.
Let $G_k[p]$ be the random spanning subgraph of $G_k$ obtained
by retaining each edge with probability $p$, independently of the others.
(Thus, if $G_k$ is complete, then $G_k[p]$ is the classical binomial random graph.)
Recently, Krivelevich, Lee and Sudakov~\cite{KLS} initiated the study of this very general model,
asking which results for classical random graphs have analogues in this setting.
As they point out, many questions do not make much sense in this generality -- for example,
no condition on $p$ can ever ensure that $G_k[p]$ is likely to contain a short cycle,
since the graphs $G_k$ may have large girth. However, some do, and indeed have positive
answers.

One of the main results in~\cite{KLS} is the following; here and later we often
suppress the dependence on $k$ in the notation. An event holds \emph{with high probability}
or \emph{whp} if its probability tends to $1$ as $k\to\infty$.

\begin{theorem}\label{th1}
Let $G_k[p]$ be defined as above, and suppose that $pk\to\infty$ as $k\to\infty$. Then whp
$G_k[p]$ contains a cycle of length at least $(1-o(1))k$.
\end{theorem}
The aim of this note is to give a short proof of this result, based
on the same basic starting point as in~\cite{KLS}, but (sadly!) using relatively simple deterministic arguments
to avoid most of the complicated probabilistic reasoning used there. Of course, the arguments
in~\cite{KLS} give significant structural information about $G_k[p]$ that cannot be deduced
from the simple proof given here.

Following~\cite{KLS}, the first step is to explore $G_p=G_k[p]$ by depth-first search,
revealing a {\em rooted spanning forest} $T$ of $G_p$. In other words, $T$ will be the disjoint
union of one or more rooted trees, one spanning each component of $G_p$. For completeness, we
describe the algorithm and its key implication, although this is of course discussed
in detail in~\cite{KLS}.

The search algorithm maintains a \emph{stack} of vertices. Initially the stack is empty,
and all vertices are \emph{unreached}. Each vertex $v$ will be `reached' at some point, and placed
onto the stack. At some later point $v$ will be removed from the stack (never to return)
and will be declared \emph{explored}. More precisely, the algorithm executes the following
loop, until every vertex is explored:
 
\smallskip
If the stack is empty, pick any unreached vertex $v$ and place it onto the stack. $v$ will be the
root of a new component of $T$.

Otherwise, let $u$ be the top vertex on the stack. If there are any
not-yet-tested edges of $G=G_k$ from $u$
to unreached vertices $w$, pick one such edge, and test whether it is present
in $G_p$. If so, add $w$ to the top of the stack, and add $uw$ to $T$. If there are no
not-yet-tested edges of $G$ from $u$ to unreached vertices, remove $u$ from the stack and label it as explored.

\smallskip
The key point is that if we find a neighbour $w$ of $u$, then we postpone looking for further
neighbours of $u$ until we have checked for neighbours of $w$, and so on. A consequence
is that the vertices on the stack always represent a path in $T$.
Let us call a path in a rooted tree \emph{vertical} if it does not contain the root
as an interior vertex. Let $U$ be the set of edges of $G_k$ not tested during the exploration.
The following two lemmas are the key starting points for the arguments of
Krivelevich, Lee and Sudakov~\cite{KLS}. We include proofs for completeness.

\begin{lemma}\label{lpath}
Every edge $e$ of $U$ joins two vertices on some vertical path in $T$.
\end{lemma}
\begin{proof}
Let $e=uv$ and suppose that $u$ was explored (left the stack) before $v$.
When $u$ left the stack, $v$ cannot have been unreached, or $uv$ would
have been tested. Also, $v$ cannot have been explored, by choice of $u$. So $v$ was on the stack,
and the stack from $v$ to $u$ forms the required vertical path.
\end{proof}

Suppressing the dependence on $k$, from now on we write $n=n(k)=|G_k|$ for the
number of vertices of $G_k$.

\begin{lemma}\label{llittle}
With high probability, at most $2n/p=o(kn)$ edges are tested during the exploration above.
\end{lemma}
\begin{proof}
Each time an edge is tested, the test succeeds (the edge is found to be present) with probability
$p$. Comparison with a binomial distribution implies that the probability
that more than $2n/p$ tests are made but fewer than $n$ succeed is $o(1)$ (in fact, exponentially small
in $n\ge k$).
But every successful test contributes an edge to the forest $T$, so at most $n-1$ tests are successful.
\end{proof}

From now on let us fix an arbitrary (small) constant $0<\eps<1/10$.
We call a vertex $v$ \emph{full} if it is incident with at least $(1-\eps)k$ edges in $U$.
\begin{corollary}\label{cfull}
With high probability, all but $o(n)$ vertices of $T$ are full.
\end{corollary}
\begin{proof}
Since $G$ has minimum degree at least $k$, each $v\in V(G)=V(T)$
that is not full is incident with at least $\eps k$ tested edges. If for some constant
$c>0$ there are at least $cn$ such vertices, then there are at least $c\eps kn/2$ tested
edges. But the probability of this is $o(1)$ by Lemma~\ref{llittle}.
\end{proof}

To state the new part of the argument, we need a couple of definitions. Let $v$ be a vertex
of a rooted forest $T$. Then there is a unique path from $v$ to the root of its component.
We write $A(v)$ for the set of \emph{ancestors} of $v$, i.e., vertices (excluding $v$) on this path.
We write $D(v)$ for the set of \emph{descendants} of $v$, again excluding $v$. Thus $w\in D(v)$ if and only
if $v\in A(w)$. The \emph{distance} $d(u,v)$ between two vertices $u$ and $v$ on a common vertical path
is just their graph distance along this path. We write $A_i(v)$ and $D_i(v)$ for the set
of ancestors/descendants of $v$ at distance exactly $i$, and $A_{\le i}(v)$, $D_{\le i}(v)$
for those at distance at most $i$. The \emph{height} of a vertex $v$ is $\max\{i: D_i(v)\ne\emptyset\}$.
 
Let us call a vertex $v$ \emph{rich} if $|D(v)|\ge \eps k$, and \emph{poor} otherwise.
In the next two lemmas,
$(T_k)_{k=1}^\infty$ is a sequence of rooted forests. We suppress the dependence on $k$ in the notation,
and write $n$ for $|T_k|$.

\begin{lemma}\label{lpoor}
Suppose that $T$ contains $o(n)$ poor vertices. Then, for any constant $C$, all but $o(n)$ vertices of $T$
are at height at least $Ck$.
\end{lemma}
\begin{proof}
For each rich vertex $v$, let $P(v)$ be a set of $\ceil{\eps k}$ descendants of $v$, obtained by choosing
vertices of $D(v)$ one-by-one starting with those furthest from $v$. For every $w\in P(v)$ we have
$D(w)\subsetneq P(v)$, so $|D(w)|<\eps k$, i.e., $w$ is poor. Consider
the set $\cS_1$ of ordered pairs $(v,w)$ with $v$ rich and $w\in P(v)$. Each of the $(1-o(1))n$ rich vertices
appears in at least $\eps k$ pairs, so $|\cS_1|\ge (1-o(1))\eps kn$.

For any vertex $w$ we have $|A_{\le i}(w)|\le i$, since
there is only one ancestor at each distance, until we hit the root.
Since $(v,w)\in \cS_1$ implies that $w$ is poor and $v\in A(w)$, and there are only $o(n)$ poor
vertices, at most $o(Ckn)=o(kn)$ pairs $(v,w)\in \cS_1$ satisfy $d(v,w)\le Ck$.
Thus $\cS_1'=\{(v,w)\in \cS_1: d(v,w)>Ck\}$ satisfies $|\cS_1'|\ge (1-o(1))\eps kn$.
Since each vertex $v$ is the first vertex of at most $\ceil{\eps k}\sim \eps k$ pairs
in $\cS_1\supset \cS_1'$, it follows that $n-o(n)$ vertices $v$ appear in pairs $(v,w)\in \cS_1'$.
Since any such $v$ has height at least $Ck$, the proof is complete.
\end{proof}

Let us call a vertex $v$ \emph{light} if $|D_{\le (1-5\eps)k}(v)|\le (1-4\eps)k$, and \emph{heavy} otherwise.
Let $H$ denote the set of heavy vertices in $T$.

\begin{lemma}\label{lmain}
Suppose that $T=T_k$ contains $o(n)$ poor vertices, and let $X\subset V(T)$ with $|X|=o(n)$.
Then, for $k$ large enough, $T$ contains a vertical path $P$ of length at least $\eps^{-2}k$ containing
at most $\eps^2k$ vertices in $X\cup H$.
\end{lemma}

\begin{proof}
Let $\cS_2$ be the set of pairs $(u,v)$ where $u$ is an ancestor of $v$ and $0<d(u,v)\le (1-5\eps)k$.
Since a vertex has at most one ancestor at any given distance, we have $|\cS_2|\le
(1-5\eps)kn$. On the other hand, by Lemma~\ref{lpoor} all but $o(n)$ vertices $u$ are at height at least $k$ and so
appear in at least $(1-5\eps)k$ pairs $(u,v)\in \cS_2$. It follows that only $o(n)$ vertices
$u$ are in more than $(1-4\eps)k$ such pairs, i.e., $|H|=o(n)$.

Let $\cS_3$ denote the set of pairs $(u,v)$ where $v\in X\cup H$, $u$ is an ancestor of $v$, and $d(u,v)\le \eps^{-2}k$.
Since a given $v$ can only appear in $\eps^{-2}k$ pairs $(u,v)\in\cS_3$,
we see that $|\cS_3|\le \eps^{-2}k|X\cup H|=o(kn)$. Hence only $o(n)$ vertices $u$ appear in more than $\eps^2 k$ pairs
$(u,v)\in \cS_3$.

By Lemma~\ref{lpoor}, all but $o(n)$ vertices are at height at least $\eps^{-2}k$. Let $u$
be such a vertex appearing in at most $\eps^2k$ pairs $(u,v)\in \cS_3$,
and let $P$ be the vertical path from $u$ to some $v\in D_{\ceil{\eps^{-2}k}}(u)$. Then $P$ has the required
properties.
\end{proof}

\noindent
\emph{Proof of Theorem~\ref{th1}.}
Fix $0<\eps<1/10$. It suffices to show that whp $G_p=G_k[p]$ contains a cycle of length at least $(1-5\eps)k$, say.

Explore $G_p=G_k[p]$ by depth-first search as described above (i.e., as in~\cite{KLS}), writing
$T$ for the spanning forest revealed, and $U$ for the set of untested edges. Condition
on the result of the exploration, noting that the edges of $U$ are still present
independently with probability $p$.
By Lemma~\ref{lpath}, $uv\in U$ implies that $u$ is either an ancestor or a descendant of $v$.
As before, let $n=|G_k|=|T|$. By Corollary~\ref{cfull}, we may assume that all but $o(n)$ vertices are full, 
i.e., are incident with at least $(1-\eps)k$ edges of $U$.

Suppose that
\begin{equation}\label{long}
 \bigl|\{ u: uv\in U,\  d(u,v)\ge (1-5\eps)k \} \bigr| \ge \eps k
\end{equation}
for some vertex $v$. Then, since $\eps k p\to\infty$, testing the relevant edges $uv$ one-by-one, 
whp we find one present in $G_p$, forming, together with $T$, the required long cycle.
Suppose then that \eqref{long} fails for every $v$.

\begin{wrapfigure}{r}{0.6in}
 \begin{center}
 \includegraphics{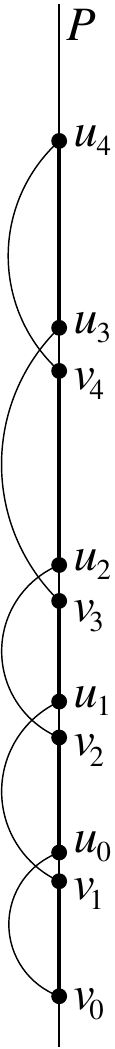}
\end{center}
\end{wrapfigure}

Suppose that some vertex $v$ is full but poor. Since $v$ has at most $\eps k$ descendants,
there are at least $(1-2\eps)k$ pairs $uv\in U$ with $u\in A(v)$. Since $v$ has only one ancestor
at each distance, it follows that \eqref{long} holds for $v$, a contradiction.

We have shown that no poor vertex is full. Hence there are $o(n)$ poor vertices,
and we may apply Lemma~\ref{lmain}, with $X$ the set of vertices that are not full.
Let $P$ be the path whose existence is guaranteed by the lemma, and
let $Z$ be the set of vertices on $P$ that are full and light, so $|V(P)\setminus Z|\le \eps^2 k$.

For any $v\in Z$, since $v$ is full, there are at least $(1-\eps)k$ vertices $u\in A(v)\cup D(v)$
with $uv\in U$. Since \eqref{long} does not hold, at least $(1-2\eps)k$ of these vertices
satisfy $d(u,v)\le (1-5\eps)k$. Since $v$ is light, in turn at least $2\eps k$
of these $u$ must be in $A(v)$. Recalling that a vertex has at most one ancestor at each
distance, we find a set $U(v)$ of at least $\eps k$ vertices $u\in A(v)$ with $uv\in U$ and
 $\eps k\le d(u,v)\le (1-5\eps)k\le k$.

It is now easy to find a (very) long cycle whp. 
Recall that $Z\subset V(P)$ with $|V(P)\setminus Z|\le \eps^2 k$.
Thinking of $P$ as oriented upwards towards the root, let $v_0$ be the lowest vertex in $Z$.
Since $|U(v_0)|\ge \eps k$ and $kp\to\infty$, whp there is an edge $v_0u_0$ in $G_p$ with $u_0\in U(v_0)$.
Let $v_1$ be the first vertex below $u_0$ along $P$ with $v_1\in Z$. Note
that we go up at least $\eps k$ steps from $v_0$ to $u_0$ and down at most $1+|V(P)\setminus Z|\le 2\eps^2 k$
from $u_0$ to $v_1$, so $v_1$ is above $v_0$.
Again whp there is an edge $v_1u_1$ in $G_p$ with $u_1\in U(v_1)$, and so at least $\eps k$ steps
above $v_1$. Continue downwards from $u_1$ to the first $v_2\in Z$
below $u_1$, and so on. Since $\eps^{-1}=O(1)$,
whp we may continue in this way to find overlapping `chords' $v_iu_i$ for $0\le i\le \floor{10\eps^{-1}}$, say.
(Note that we remain within $P$ as each upwards step has length at most $k$.) These chords
combine with $P$ to form a cycle $C$, as shown in the figure. Since each chord $v_iu_i$ corresponds
to at least $\eps k$ steps up $P$, and each of the segments $v_iu_{i-1}$ of $P$ has
length at most $2\eps^2k$, we see that all of the $\floor{10\eps^{-1}}+1\ge 10\eps^{-1}$ segments of $P$ included
in $C$ have length at least $(\eps-4\eps^2)k\ge \eps k/2$, so $C$
has length at least $5k$.
\hfill\endproof


\begin{thebibliography}{9}

\bibitem{KLS} M.~Krivelevich, C.~Lee and B.~Sudakov,
 Long paths and cycles in random subgraphs of graphs with large minimum degree,
 to appear in \emph{Random Structures and Algorithms} (doi: 10.1002/rsa.20508).
 arXiv:1207.0312v3

\end{thebibliography}
\end{document}